\documentclass[12pt]{article}
\usepackage{amsthm,amssymb,amsmath,graphicx,cite}
\usepackage{algorithmic, algorithm}
\usepackage[margin = 1in]{geometry}

\newcommand{\LHS}{\mathsf{LHS}}
\newcommand{\RHS}{\mathsf{RHS}}

\newcommand{\Oh}{{\mathcal O}}
\newcommand{\dmin}{\displaystyle\min}

\newcommand{\ip}[2]{\left\langle #1 , #2 \right\rangle}    
\newcommand{\R}{{\mathbb R}}

\newcommand{\prox}{{\text{\sf Prox}}}

\DeclareMathOperator*{\argmin}{arg\,min}

\DeclareMathOperator{\dist}{dist}

\newtheorem{theorem}{Theorem}
\newtheorem{proposition}{Proposition}

\title{Convergence rates of proximal gradient methods via the convex conjugate}
\author{
David H. Gutman\thanks{Department of Mathematical Sciences, 
Carnegie Mellon University, USA, {\tt dgutman@andrew.cmu.edu}}
 \and
Javier F. Pe\~na\thanks{Tepper School of Business,
Carnegie Mellon University, USA, {\tt jfp@andrew.cmu.edu}}
}

\begin{document}

\maketitle

\begin{abstract}
We give a novel proof of the $\Oh(1/k)$ and $\Oh(1/k^2)$ convergence rates of the proximal gradient and accelerated proximal gradient methods for composite convex minimization.  
The crux of the new proof is an upper bound constructed via the convex conjugate of the objective function.  
\end{abstract}


\section{Introduction}
\label{sec.intro}

The development of accelerated versions of first-order methods has had a profound influence in convex optimization.  In his seminal paper~\cite{Nest83} Nesterov devised a first-order algorithm with optimal $\Oh(1/k^2)$ rate of convergence for unconstrained convex optimization 
via a modification of the standard gradient descent algorithm that includes {\em momentum} steps.  A later breakthrough was the acceleration of the {\em proximal gradient method} independently developed by Beck and Teboulle~\cite{BeckT09} and by~Nesterov~\cite{Nest13}.  
The proximal gradient method, also known as the forward-backward method~\cite{LionM79}, is an extension of the gradient descent method to solve the composite minimization problem
\begin{equation}\label{eq.problem.composite}
\min_{x\in \R^n} \varphi(x) + \psi(x)
\end{equation}
where $\varphi:\R^n\rightarrow \R$ is differentiable and $\psi: \R^n\rightarrow \R\cup\{\infty\}$ is a closed convex function such that for $t > 0$ the proximal map
\begin{equation}\label{eq.prox.map}
\prox_{t}(x):= \argmin_{y\in \R^n}\left\{\psi(y) + \frac{1}{2t}\|x-y\|^2\right\}
\end{equation}
is computable.

The significance of Nesterov's and Beck and Teboulle's breakthroughs has prompted interest in new approaches to explain how acceleration is achieved in first-order methods~\cite{AlleO14,BubeLS15,DrusFR16,FlamB15,LessRP16,Pena17,SuBC14}.  Some of these approaches are based on geometric~\cite{BubeLS15,DrusFR16}, control~\cite{LessRP16}, and differential equations~\cite{SuBC14} techniques.   The recent article~\cite{Pena17} 
relies on the convex conjugate  to give  a unified and succinct derivation of the $\Oh(1/\sqrt{k}), \Oh(1/k),$ and $\Oh(1/k^2)$ convergence rates of the subgradient, gradient, and accelerated gradient methods for unconstrained smooth convex minimization.  The crux of the approach in~\cite{Pena17} is a generic upper bound on the iterates generated by the subgradient, gradient, and accelerated gradient algorithms constructed via the convex conjugate of the objective function.

We extend the main construction in~\cite{Pena17} to give a unified derivation of the convergence rates of the proximal gradient and accelerated proximal gradient algorithms for the composite convex minimization problem \eqref{eq.problem.composite}.  As in~\cite{Pena17}, the central result of this paper (Theorem~\ref{main.thm}) is an upper bound on the iterates 
generated by both the non-accelerated and the accelerated proximal gradient methods.  This bound is constructed via the convex conjugate of the objective function.  Theorem~\ref{main.thm} readily yields the widely known $\Oh(1/k)$ and $\Oh(1/k^2)$ convergence rates of the proximal gradient and accelerated proximal gradient algorithms for \eqref{eq.problem.composite} when the smooth component $\varphi$ has Lipschitz gradient and the step sizes are chosen judiciously.  Theorem~\ref{main.thm} highlights some key similarities and differences between the non-accelerated and the accelerated algorithms.  It is noteworthy that Theorem~\ref{main.thm} and its variant, Theorem~\ref{thm.two}, hold under certain conditions on the step sizes and momentum used in the algorithm but do not require any Lipschitz assumption.  The convex conjugate approach underlying Theorem~\ref{main.thm} also extends to a {\em proximal subgradient algorithm} when the component $\varphi$ is merely convex but not necessarily smooth.  (See Algorithm~\ref{algo.subgrad} and Proposition~\ref{prop.subgrad}.)  This extension automatically yields a novel derivation of both classical~\cite[Theorem 3.2.2]{Nest04} as well as modern convergence rates~\cite[Theorem 5]{Grim17} for the projected subgradient algorithm.  The latter derivations are similar to the derivation of the convergence rates for the proximal gradient and accelerated proximal gradient algorithms.

\medskip

Throughout the paper we assume that $\R^n$ is endowed with an inner product $\ip{\cdot}{\cdot}$ and that $\|\cdot\|$ denotes the corresponding Euclidean norm.

\section{Proximal gradient and accelerated proximal gradient methods}
\label{sec.main}
Let $\varphi:\R^n \rightarrow \R$ be a differentiable convex function and $\psi:\R^n \rightarrow \R\cup\{\infty\}$ be a closed convex function such that the proximal map~\eqref{eq.prox.map} is computable.  Let $f := \varphi + \psi$ and consider the problem~\eqref{eq.problem.composite} that can be rewritten as
\begin{equation}\label{eq.problem}
\dmin_{x\in \R^n}  f(x). 
\end{equation} 
Algorithm~\ref{algo.prox} describes a template of a proximal gradient algorithm for \eqref{eq.problem}.  

\begin{algorithm}
\caption{Template for proximal gradient method}\label{algo.prox}
\begin{algorithmic}[1]
	\STATE {\bf input:}  $x_0 \in \R^n$
	\STATE $y_0 := x_0; \; \theta_0 := 1$ 
	\FOR{$k=0,1,2,\dots$}
		\STATE pick $t_k > 0$ 
		\STATE $x_{k+1}:=\prox_{t_k}(y_k-t_k \nabla \varphi(y_k))$
		\STATE pick $\theta_{k+1} \in (0,1]$
		\STATE $y_{k+1}:=x_{k+1} + \frac{\theta_{k+1}(1-\theta_k)}{\theta_k} (x_{k+1}-x_k)$
	\ENDFOR
\end{algorithmic}
\end{algorithm}

Step 7 of Algorithm~\ref{algo.prox} incorporates a momentum step.
The  (non-accelerated) proximal gradient method is obtained by choosing $\theta_{k+1} = 1$ in Step 6.  In this case Step 7 simply sets $y_{k+1} = x_{k+1}$ and  does not incorporate any momentum.  Other choices of $\theta_{k+1}\in (0,1]$ yield accelerated versions of the proximal gradient method.  In particular, the FISTA algorithm in~\cite{BeckT09} is obtained by choosing $\theta_{k+1} \in (0,1]$ via the rule
$
\theta_{k+1}^2 = \theta_k^2 (1-\theta_{k+1}).
$ In this case $\theta_k \in (0,1)$ for $k \ge 1$ and there is a non-trivial momentum term in Step 7.

The main result in this paper is Theorem~\ref{main.thm} below which subsumes the widely known convergence rates $\Oh(1/k)$ and $\Oh(1/k^2)$ of the proximal gradient and accelerated proximal gradient algorithms under suitable choices of $t_k, \theta_k, \; k=0,1,\dots$.    


Theorem~\ref{main.thm} relies on a suitable constructed sequence $z_k \in \R^n, \; k=1,2,\dots.$ 
The  construction of $z_k \in \R^n, \; k=1,2,\dots$ in turn is motivated by the identity~\eqref{eq.kstep.crux} below.  

Consider Step 5 in Algorithm~\ref{algo.prox}, namely
\begin{equation}\label{eq.prox.step}
x_{k+1} = \prox_{t_k}(y_k - t_k \nabla \varphi(y_k)).
\end{equation}
The optimality conditions for~\eqref{eq.prox.step} imply that 
\[
x_{k+1} = y_k - t_k \cdot g_k 
\]
where $g_k := g^\varphi_k + g^\psi_k$ for $g^\varphi_k:=\nabla \varphi(y_k)$ and for some $g^\psi_k \in \partial \psi(x_{k+1}).$  

Step 5 and Step 7 of Algorithm~\ref{algo.prox} imply that for $k=0,1,\dots$
\[
\frac{y_{k+1} - (1-\theta_{k+1})x_{k+1}}{\theta_{k+1}} = \frac{x_{k+1} - (1-\theta_{k})x_{k}}{\theta_{k}} = \frac{y_{k} - (1-\theta_{k})x_{k}}{\theta_{k}} - \frac{t_k}{\theta_k}g_k.
\]
Since $\theta_0 = 1$ and $y_0=x_0$, it follows that for $k=1,2,\dots$ 
\begin{equation}\label{eq.kstep.crux}
\frac{y_{k} - (1-\theta_k)x_{k}}{\theta_k} = x_0 - \sum_{i=0}^{k-1} 
\frac{t_i}{\theta_i}g_i
\Leftrightarrow
(1-\theta_k)(y_{k} - x_{k}) = \theta_{k}\left(x_0 - y_k - \sum_{i=0}^{k-1} 
\frac{t_i}{\theta_i}g_i\right).
\end{equation}

As it is customary, we will assume that the step sizes $t_k$ chosen at  Step 4 in Algorithm~\ref{algo.prox} satisfy the following decrease condition
\begin{align}\label{eq.decrease}
f(x_{k+1}) &\le \min_{x\in \R^n}\left\{\varphi(y_k) + \ip{\nabla \varphi(y_k)}{x-y_k}+  \frac{1}{2t_k}\|x-y_k\|^2 + \psi(x)\right\} \notag\\
&= 
\varphi(y_k) + \psi(x_{k+1})+ \ip{g^\psi_k}{y_k-x_{k+1}} - \frac{t_k}{2} \|g_k\|^2.
\end{align}   
The condition~\eqref{eq.decrease} holds in particular when  $\nabla \varphi $  is Lipschitz and $t_k, \; k=0,1,\dots$ are chosen via a standard  backtracking procedure.
Observe that \eqref{eq.decrease} implies $f(x_{k+1}) \le f(y_k)$.

Theorem~\ref{main.thm} also relies on the convex conjugate function.  Recall that if $h:\R^n \rightarrow \R\cup\{\infty\}$ is a convex function then its {\em convex conjugate} $h^*:\R^n\rightarrow \R\cup \{\infty\}$ is defined as
\[
h^*(z) = \sup_{x\in \R^n} \left\{\ip{z}{x} - h(x)\right\}.
\]

\begin{theorem}\label{main.thm} 
Suppose $\theta_k \in (0,1], \; k=0,1,2,\dots$ and the step sizes $t_k>0,\; k=0,1,2,\dots$ are such that~\eqref{eq.decrease} holds.
Let $x_k \in \R^n, \; k=1,2,\dots$ be the iterates generated by Algorithm~\ref{algo.prox}.  Let $z_k\in \R^n, \; k = 1,2\dots$ be as follows
\begin{equation}\label{eq.construction}
z_k := \frac{\displaystyle\sum_{i=0}^{k-1} \frac{t_i}{\theta_i} g_i}{\displaystyle\sum_{i=0}^{k-1} \frac{t_i}{\theta_i}}. 
\end{equation}
Then
\begin{equation}\label{eq.thm.grad}
\LHS_k  \le  -f^*(z_k)+ \ip{z_k}{x_0} - 
\frac{\sum_{i=0}^{k-1} \frac{t_i}{\theta_i}}{2} \|z_k\|^2,
\end{equation}
where $\LHS_k$ is as follows depending on the choice of $\theta_k \in (0,1]$ and $t_k>0$.
\begin{itemize}
\item[(a)] When $\theta_k = 1, k=0,1,\dots$ let
\[
\LHS_k := \frac{\sum_{i=0}^{k}t_i f(x_{i+1})}{\sum_{i=0}^{k}t_i}.
\]
\item[(b)] When $t_k>0$ and $\theta_k \in (0,1], \; k=0,1,\dots$ are such that $\sum_{i=0}^{k-1} \frac{t_i}{\theta_i} = (1-\theta_k)\sum_{i=0}^{k} \frac{t_i}{\theta_i}$  let
\[
\LHS_k = f(x_k).
\]
\end{itemize}
\end{theorem}

\medskip

Theorem~\ref{main.thm} readily implies that in both case (a) and case (b)
\begin{align*}
\LHS_k  &\le  \min_{u\in \R^n} \left\{f(u) - \ip{z_k}{u}\right\} 
 + \min_{u\in\R^n}\left\{\ip{z_k}{u} + \frac{1}{2\cdot \sum_{i=0}^{k-1} \frac{t_i}{\theta_i}}\|u-x_0\|^2  \right\} \\
 & \le \min_{u\in \R^n} \left\{ f(u) + \frac{1}{2\cdot \sum_{i=0}^{k-1} \frac{t_i}{\theta_i}}\|u-x_0\|^2  \right\} \\
 &\le f(x) +  \frac{1}{2\cdot \sum_{i=0}^{k-1} \frac{t_i}{\theta_i}}\|x-x_0\|^2
\end{align*}
for all $x\in \R^n$.  

Let $\bar f$ and $\bar X$ respectively denote the optimal value and set of optimal solutions to \eqref{eq.problem}.  If $\bar f$ is finite and $\bar X$ is nonempty then in both case (a) and case (b) of Theorem~\ref{main.thm} we  get
\begin{equation}\label{eq.ineq.proximal}
f(x_k) - \bar f \le \frac{\dist(x_0,\bar X)^2}{2\cdot \sum_{i=0}^{k-1} \frac{t_i}{\theta_i}}.
\end{equation}
Suppose  $t_k \ge \frac{1}{L}, \; k=0,1,2,\dots$ for some constant $L>0$.  This holds in particular if $\nabla \varphi$ is Lipschitz and $t_k$ is chosen via a standard backtracking procedure.  Then inequality \eqref{eq.ineq.proximal} yields the following known convergence bound for the proximal gradient method
\[
f(x_k) - \bar f \le \frac{L\cdot\dist(x_0,\bar X)^2}{2k}.
\]
On the other hand, suppose $t_k = \frac{1}{L}, \; k=0,1,2,\dots$ for some constant $L>0$ and $\theta_k, \; k=0,1,2,\dots$ are chosen via $\theta_0 = 1$ and $\theta_{k+1}^2 = \theta_k^2(1-\theta_{k+1}).$  Then a straightforward induction shows that 
$$\sum_{i=0}^{k-1}\frac{t_i}{\theta_i}  =(1-\theta_k) \sum_{i=0}^{k}\frac{t_i}{\theta_i} =\frac{1}{L \theta_{k-1}^2} 
\ge \frac{(k+1)^2}{4L}.$$
Thus case (b) in Theorem~\ref{main.thm} applies and inequality \eqref{eq.ineq.proximal} yields the following known convergence bound for the accelerated proximal gradient method
\[
f(x_k) - \bar f \le \frac{2L\cdot\dist(x_0,\bar X)^2}{(k+1)^2}.
\]
Although Theorem~\ref{main.thm} yields the iconic $\Oh(1/k^2)$ convergence rate of the accelerated proximal gradient algorithm, it applies under the somewhat restrictive conditions stated in case (b) above.  In particular, case (b) does not cover the more general case when $t_k, \; k=0,1,\dots$ are chosen via backtracking as in the FISTA with backtracking algorithm in~\cite{BeckT09}.  
The convergence rate in this case, namely~\cite[Theorem 4.4]{BeckT09} is a consequence of Theorem~\ref{thm.two} below.   Theorem~\ref{thm.two} is a variant of Theorem~\ref{main.thm}(b) that applies to more flexible choices of $t_k,\theta_k, \; k=0,1,\dots$.  In particular, Theorem~\ref{thm.two} applies to the popular choice $\theta_k = \frac{2}{k+2}, \; k=0,1,\dots$.
\begin{theorem}\label{thm.two} Suppose $\bar f = \dmin_{x\in\R^n} f(x)$ is finite,  $\theta_k\in(0,1], \; k=0,1,2,\dots$ satisfy $\theta_0 = 1$ and $\theta_{k+1}^2 \ge \theta_k^2(1-\theta_{k+1}),$ and the step sizes $t_k>0,\; k=0,1,2,\dots$ are non-increasing and such that~\eqref{eq.decrease} holds.  Let $x_k \in \R^n, \; k=1,2,\dots$ be the iterates generated by Algorithm~\ref{algo.prox}.  Let $z_k \in \R^n, \; k=1,2,\dots$ be as follows
\[
z_k = \frac{\theta_{k-1}^2}{t_{k-1}} \cdot \displaystyle\sum_{i=0}^{k-1} \frac{t_i}{\theta_i} g_i. 
\]
Then for $k=1,2,\dots$ 
\begin{equation}\label{eq.thm.two}
f(x_k) - \bar f  \le  -(R_k\cdot (f - \bar f))^*(z_k) + \ip{z_k}{x_0} - \frac{t_{k-1}}{2\theta_{k-1}^2}\|z_k\|^2,
\end{equation}
where $R_1 = 1$ and $R_{k+1} = \frac{t_{k-1}}{t_{k}} \cdot \frac{\theta_{k}^2}{\theta_{k-1}^2(1-\theta_k)} \cdot R_k \ge 1, \; k=1,2,\dots.$  In particular, if $\bar X = \{x\in \R^n: f(x) = \bar f\}$ is nonempty then
\[
f(x_k) - \bar f \le \min_{u\in\R^n}\left\{R_k \cdot(f(u)-\bar f) + \frac{\theta_{k-1}^2}{2t_{k-1}} \|u-x_0\|^2 \right\} =\frac{\theta_{k-1}^2\cdot \dist(x_0,\bar X)^2}{2t_{k-1}}.
\]
\end{theorem}

Suppose the step sizes $t_k, \; k=0,1,2,\dots$ are non-increasing, satisfy~\eqref{eq.decrease}, and $t_k \ge \frac{1}{L}, \; k=0,1,2,\dots$ for some constant $L > 0$.  This holds in particular when
$\nabla \varphi$ is Lipschitz and $t_k$ is chosen via a suitable backtracking procedure as the one in~\cite{BeckT09}.  If $\theta_0 = 1$ and $\theta_{k+1}^2 \ge \theta_k^2(1-\theta_{k+1}),\; k=0,1,\dots$ then Theorem~\ref{thm.two} implies that
\[
f(x_k) - \bar f \le \frac{L \theta_{k-1}^2 \cdot \dist(x_0,\bar X)^2}{2}.
\]
Hence if $\theta_{k+1}^2 = \theta_k^2(1-\theta_{k+1})$ or $\theta_k = \frac{2}{k+2}$ for $k=0,1,\dots$ then
\[
f(x_k) - \bar f \le \frac{2 L  \cdot \dist(x_0,\bar X)^2}{(k+1)^2}.
\]

\section{Proof of Theorem~\ref{main.thm} and Theorem~\ref{thm.two}}
\label{sec.proof}
We will use the following properties of the convex conjugate.  

Suppose $h:\R^n \rightarrow\R\cup\{\infty\}$ is a convex function.  Then 
\begin{equation}\label{eq.property.0}
h^*(z) + h(x) \ge \ip{z}{x}
\end{equation}
for all $z,x\in \R^n,$ and equality holds if $z \in \partial h(x)$.  
 
Suppose $f,\varphi,\psi:\R^n \rightarrow\R\cup\{\infty\}$ are convex functions and $f = \varphi+\psi$.  Then 
\begin{equation}\label{eq.property.1}
 f^*(z^\varphi+z^\psi) \le \varphi^*(z^\varphi) + \psi^*(z^\psi) \; \text{ for all } z^\varphi, z^\psi \in \R^n.
\end{equation}

Suppose $f:\R^n\rightarrow \R_+\cup\{\infty\}$ is a convex function and $R\ge 1$. Then
\begin{equation}\label{eq.property.2}
(R\cdot f)^*(Rz) = R \cdot (f^*(z)),
\end{equation}
and
\begin{equation}\label{eq.property.3}
(R\cdot f)^*(z) \le  f^*(z).
\end{equation}

\subsection{Proof of Theorem~\ref{main.thm}}
We prove~\eqref{eq.thm.grad} by induction.  To ease notation, let $\mu_k:= \frac{1}{\sum_{i=0}^{k-1} \frac{t_i}{\theta_i}}$ throughout this proof.  For $k=1$ we have
\begin{align*}
\LHS_1 = f(x_1) 
&\le \varphi(x_0) + \psi(x_1) + \ip{g^\psi_0}{x_0-x_1} - \frac{t_0}{2}\|g_0\|^2  \\
  &= \varphi(x_0) - \ip{g_0^\varphi}{x_0} + \psi(x_1) -\ip{g_0^\psi}{x_1}
 + \ip{g_0}{x_0} - \frac{t_0}{2}\|g_0\|^2 \\
 & = -\varphi^*(g_0^\varphi) - \psi^*(g_0^\psi) + \ip{g_0}{x_0} - \frac{t_0}{2}\|g_0\|^2\\
 & \le -f^*(z_1) + \ip{z_1}{x_0} - \frac{\|z_1\|^2}{2\mu_1}. 
\end{align*}
The first step follows from~\eqref{eq.decrease}. The third step follows from~\eqref{eq.property.0} and $g^\varphi_0 = \nabla \varphi(x_0), \; g^\psi_0 \in \partial \psi(x_1)$. The last step follows from~\eqref{eq.property.1} and the choice of $z_1 = g_0 = g^\varphi_0+ g^\psi_0$ and $\mu_1 = \frac{1}{t_0}$.

Suppose~\eqref{eq.thm.grad} holds for $k$ and let $\gamma_k = \frac{t_k/\theta_k}{\sum_{i=0}^{k} t_i/\theta_i}$.  The construction~\eqref{eq.construction} implies that
\begin{align*}
z_{k+1} &= (1-\gamma_k) z_k + \gamma_k g_k \\
\mu_{k+1} &=(1-\gamma_k)\mu_k.
\end{align*}
Therefore,
\begin{equation}\label{eq.ind.1}
\ip{z_{k+1}}{x_0} - \frac{\|z_{k+1}\|^2}{2\mu_{k+1}} = (1-\gamma_k) \left(\ip{z_{k}}{x_0} - \frac{\|z_{k}\|^2}{2\mu_{k}}\right)
 +  \gamma_k \left(\ip{g_k}{x_0-\frac{z_k}{\mu_k}} -\frac{\gamma_k}{2(1-\gamma_k) \mu_k} \|g_k\|^2\right).
\end{equation}
In addition, the convexity of $f^*$, properties~\eqref{eq.property.0}, ~\eqref{eq.property.1}, and  $g^\varphi_k = \nabla \varphi(y_k),\; g^\psi_k \in \partial \psi(x_{k+1}), \; g_k = g^\varphi_k + g^\psi_k$ imply 
\begin{align}\label{eq.ind.2}
-f^*(z_{k+1}) &\ge -(1-\gamma_k)f^*(z_k) - \gamma_k f^*(g_k) \notag \\
&\ge -(1-\gamma_k)f^*(z_k) - \gamma_k (\varphi^*(g^\varphi_k) + \psi^*(g^\psi_k)) \\
&= -(1-\gamma_k)f^*(z_k) - \gamma_k \left(\ip{g^\varphi_k}{y_{k}} - \varphi(y_{k}) + \ip{g^\psi_k}{x_{k+1}} - \psi(x_{k+1})\right).\notag
\end{align}
Let $\RHS_k$ denote the right-hand side in~\eqref{eq.thm.grad}.  From~\eqref{eq.ind.1} and \eqref{eq.ind.2} it follows that
\begin{align}\label{eq.induction}
\RHS&_{k+1} - (1-\gamma_k)\RHS_k \\
& \ge \gamma_k \left(\ip{g_{k}}{x_0-y_k-\frac{z_{k}}{\mu_k}} + \varphi(y_k)+\psi(x_{k+1}) + \ip{g^\psi_{k}}{y_k - x_{k+1}}-\frac{\gamma_k}{2(1-\gamma_k) \mu_k} \|g_{k}\|^2\right).\notag
\end{align}
Hence to complete the proof of~\eqref{eq.thm.grad} by induction it suffices to show that
\begin{align}\label{eq.to.show}
\LHS&_{k+1} - (1-\gamma_k)\LHS_k \\
& \le \gamma_k \left(\ip{g_{k}}{x_0-y_k-\frac{z_{k}}{\mu_k}} + \varphi(y_k)+\psi(x_{k+1}) + \ip{g^\psi_{k}}{y_k - x_{k+1}}-\frac{\gamma_k}{2(1-\gamma_k) \mu_k} \|g_{k}\|^2\right).\notag
\end{align}
To that end, we consider case (a) and case (b) separately.

\medskip

\noindent
{\bf Case (a).}  In this case $\gamma_k = \frac{t_k}{\sum_{i=0}^k t_i}$ and $y_k=x_{k}.$
Thus $\mu_k = \frac{1}{\sum_{i=0}^{k-1} t_i},\; \frac{\gamma_k}{(1-\gamma_k)\mu_k} =t_k,$ and $x_0-y_k - \frac{z_k}{\mu_k} = 0$.  
Therefore
\begin{align*}
\LHS_{k+1} &- (1-\gamma_k)\LHS_k \\ 
&= \gamma_k \cdot f(x_{k+1}) \\
&\le \gamma_k\left(\varphi(y_{k})+ \psi(x_{k+1}) + \ip{g^\psi_k}{y_k-x_{k+1}}-\frac{t_k}{2}\|g_k\|^2\right)
\\
&= 
\gamma_k\left(\varphi(y_{k}) + \psi(x_{k+1}) + \ip{g^\psi_{k}}{y_k - x_{k+1}}- \frac{\gamma_k}{2(1-\gamma_k) \mu_k} \|g_k\|^2\right) 
\\&=\gamma_k \left(\ip{g_{k}}{x_0-y_k-\frac{z_{k}}{\mu_k}} + \varphi(y_k)+\psi(x_{k+1}) + \ip{g^\psi_{k}}{y_k - x_{k+1}}-\frac{\gamma_k}{2(1-\gamma_k) \mu_k} \|g_{k}\|^2\right).
\end{align*}
The second step follows from~\eqref{eq.decrease}. The third and fourth steps follow from $\frac{\gamma_k}{(1-\gamma_k)\mu_k} = t_k$ and $x_0-y_k - \frac{z_k}{\mu_k} = 0$ respectively.  Thus~\eqref{eq.to.show} holds in case (a).

\medskip

\noindent
{\bf Case (b).}  In this case $\gamma_k =\theta_k$ and $\frac{\gamma_k^2}{(1-\gamma_k)\mu_k} = t_k$.  
Therefore
\begin{align*}
\LHS_{k+1} &- (1-\gamma_k)\LHS_k \\ 
 &= f(x_{k+1}) - (1-\gamma_k)(\varphi(x_{k}) + \psi(x_{k})) \\
&\le 
\varphi(y_k) + \psi(x_{k+1})+ \ip{g^\psi_k}{y_k-x_{k+1}} -\frac{t_k}{2}\|g_k\|^2 \\
&\;\;\;- (1-\gamma_k)\left(\varphi(y_k) + \ip{g^\varphi_k}{x_k-y_k} + \psi(x_{k+1}) + \ip{g^\psi_k}{x_k-x_{k+1}}\right)
\\
&= \gamma_k\left(\varphi(y_k) + \psi(x_{k+1}) + \ip{g^\psi_k}{y_k-x_{k+1}}\right) 
+(1-\gamma_k) \ip{g_k}{y_k-x_k} - \frac{t_k}{2} \|g_k\|^2
\\
&=
\gamma_k \left(\ip{g_{k}}{x_0-y_k-\frac{z_{k}}{\mu_k}} + \varphi(y_k)+\psi(x_{k+1}) + \ip{g^\psi_{k}}{y_k - x_{k+1}}-\frac{\gamma_k}{2(1-\gamma_k) \mu_k} \|g_{k}\|^2\right).
\end{align*}
The second step follows from~\eqref{eq.decrease} and the convexity of $\varphi$ and $\psi$.  The last step follows from~$\theta_k = \gamma_k,$~equation~\eqref{eq.kstep.crux},  and~$\frac{\gamma_k^2}{(1-\gamma_k)\mu_k} = t_k.$  Thus~\eqref{eq.to.show} holds in case (b) as well.

\subsection{Proof of Theorem~\ref{thm.two}}

The proof of Theorem~\ref{thm.two} is a modification of the proof of Theorem~\ref{main.thm}.  Without loss of generality assume $\bar f = 0$ as otherwise we can work with $f-\bar f$ in place of $f$.  
Again we prove~\eqref{eq.thm.two} by induction. To ease notation, let $\mu_k:= \frac{\theta_{k-1}^2}{t_{k-1}}$ throughout this proof.
For $k=1$ inequality \eqref{eq.thm.two} is identical to \eqref{eq.thm.grad} since $R_1 = 1$ and $\theta_0 = 1$.  Hence this case follows from the proof of Theorem~\ref{main.thm} for $k=1$.    Suppose~\eqref{eq.thm.two} holds for $k$.  Observe that
\begin{align*}
z_{k+1} &= \rho_k (1-\theta_k)z_k + \theta_k g_k \\
\mu_{k+1} &= \rho_k (1-\theta_k)\mu_k
\end{align*}
for $\rho_k := \frac{R_{k+1}}{R_k}  = \frac{t_{k-1}}{t_k} \cdot\frac{\theta_k^2}{\theta_{k-1}^2(1-\theta_{k})} = \frac{\mu_{k+1}}{\mu_k(1-\theta_{k})} \ge 1$.  
Next, proceed as in the proof of Theorem~\ref{main.thm}.   First, 
\begin{align}\label{eq.ind.1.two}
\ip{z_{k+1}}{x_0} - \frac{\|z_{k+1}\|^2}{2\mu_{k+1}} &= \rho_k(1-\theta_k) \left(\ip{z_{k}}{x_0} - \frac{\|z_{k}\|^2}{2\mu_{k}}\right)
 +  \theta_k \cdot \ip{g_k}{x_0-\frac{z_k}{\mu_k}} -\frac{\theta_k^2}{2 \mu_{k+1}} \|g_k\|^2\notag\\
&= \rho_k(1-\theta_k) \left(\ip{z_{k}}{x_0} - \frac{\|z_{k}\|^2}{2\mu_{k}}\right)
 +  \theta_k \cdot \ip{g_k}{x_0-\frac{z_k}{\mu_k}} -\frac{t_k}{2} \|g_k\|^2.
\end{align}
Second, the convexity of $f^*$ and the fact that $f \ge \bar f = 0$ imply
\begin{align}\label{eq.ind.2.two}
-(R_{k+1} \cdot f)^*(z_{k+1}) &\ge -(1-\theta_k)(R_{k+1} \cdot f)^*(\rho_k\cdot z_k) - \theta_k (R_{k+1} \cdot f)^*(g_k) \notag \\
&\ge -(1-\theta_k)(\rho_k\cdot R_k \cdot f)^*(\rho_k\cdot z_k) - \theta_k \cdot f^*(g_k)
\\ & \ge
-\rho_k (1-\theta_k)(R_k\cdot f)^*(z_k)-\theta_k (\varphi^*(g^\varphi_k) + \psi^*(g^\psi_k)) \notag\\
&= -\rho_k(1-\theta_k)(R_k\cdot f)^*(z_k) - \theta_k \left(\ip{g^\varphi_k}{y_{k}} - \varphi(y_{k}) + \ip{g^\psi_k}{x_{k+1}} - \psi(x_{k+1})\right).\notag
\end{align}
The first step follows from the convexity of $f^*$.  The second step follows from~\eqref{eq.property.3}.  The third step follows from~\eqref{eq.property.1} and~\eqref{eq.property.2}.  The last step follows from~\eqref{eq.property.0} and $g^\varphi_k = \nabla \varphi(y_k),\; g^\psi_k \in \partial \psi(x_{k+1}).$ 

Let $\RHS_k$ denote the right-hand side in~\eqref{eq.thm.two}.  The induction hypothesis implies that $\RHS_k \ge f(x_k) \ge 0$.  Thus from \eqref{eq.ind.1.two}, \eqref{eq.ind.2.two}, and $\rho_k \ge 1$ it follows that
\begin{align}\label{eq.induction.two}
\RHS&_{k+1} - (1-\theta_k)\RHS_k \notag\\
& \ge \RHS_{k+1} - \rho_k(1-\theta_k)\RHS_k \\
& \ge \theta_k \left(\ip{g_{k}}{x_0-y_k-\frac{z_{k}}{\mu_k}} + \varphi(y_k)+\psi(x_{k+1}) + \ip{g^\psi_{k}}{y_k - x_{k+1}}\right)-\frac{t_k}{2} \|g_{k}\|^2.\notag
\end{align}
Finally, proceeding exactly as in case (b) in the proof of Theorem~\ref{main.thm} we get
\begin{align*}
f(x_{k+1}) &- (1-\theta_k) f(x_k) \\
&\le \theta_k\left(\varphi(y_k) + \psi(x_{k+1}) + \ip{g^\psi_k}{y_k-x_{k+1}}\right) 
+(1-\theta_k) \ip{g_k}{y_k-x_k} - \frac{t_k}{2} \|g_k\|^2
\\
&=
\theta_k \left(\ip{g_{k}}{x_0-y_k-\frac{z_{k}}{\mu_k}} + \varphi(y_k)+\psi(x_{k+1}) + \ip{g^\psi_{k}}{y_k - x_{k+1}}\right)-\frac{t_k}{2} \|g_{k}\|^2 \\
&\le \RHS_{k+1} - (1-\theta_k)\RHS_k.
\end{align*}
The second step follows from~\eqref{eq.kstep.crux}.  The third step follows from~\eqref{eq.induction.two}.   This completes the proof by induction.

\section{Proximal subgradient method}
Algorithm~\ref{algo.subgrad} describes a variant of Algorithm~\ref{algo.prox} for the case when $\varphi:\R^n \rightarrow \R$ is merely convex. 

\begin{algorithm}
\caption{Proximal subgradient method}\label{algo.subgrad}
\begin{algorithmic}[1]
	\STATE {\bf input:}  $x_0 \in \R^n$
	\FOR{$k=0,1,2,\dots$}
		\STATE pick $g^\varphi_k \in \partial \varphi(x_k)$ and $t_k > 0$
		\STATE $x_{k+1}:=\prox_{t_k}(x_k-t_kg^\varphi_k)$
	\ENDFOR
\end{algorithmic}
\end{algorithm}

When $\psi$ is the indicator function $I_C$ of a closed convex set $C$, Step 4 in Algorithm~\ref{algo.subgrad} can be rewritten as $x_{k+1} = \displaystyle\argmin_{x\in C} \|x_k - t_k \cdot g^\varphi_k - x\| = \Pi_C(x_k - t_k \cdot g^\varphi_k)$.  Hence when $\psi =I_C$ Algorithm~\ref{algo.subgrad} becomes the projected subgradient method for
\begin{equation}\label{eq.problem.const}
\min_{x\in C} \varphi(x).
\end{equation}
The classical convergence rate for the projected gradient is an immediate consequence of Proposition~\ref{prop.subgrad} as we detail below.   Proposition~\ref{prop.subgrad} in turn is obtained via a minor tweak on the construction and proof of Theorem~\ref{main.thm}.
Observe that 
\[
x_{k+1} = \prox_{t_k}(x_k - t_k g^\varphi_k) \Leftrightarrow x_{k+1} = x_k -t_k \cdot g_k 
\]
where $g_k = g^\varphi_k+g^\psi_k$ for some $g^\psi_k \in \partial \psi(x_{k+1}).$ 
Next,  let $z_k \in \R^n,
\; k = 0,1,2\dots$ be as follows
\begin{equation}\label{eq.sub.1}
z_{k} = \frac{\sum_{i=0}^{k} t_i g_i}{\sum_{i=0}^{k} t_i}.
\end{equation}

\begin{proposition}\label{prop.subgrad}
\label{thm.subgrad}  Let $x_k \in \R^n, \; k=0,1,2,\dots$ be the sequence of iterates generated by Algorithm~\ref{algo.subgrad} and let $z_k \in \R^n,
\; k = 0,1,2\dots$ be defined by~\eqref{eq.sub.1}.  Then for $k=0,1,2,\dots$
\begin{align}\label{eq.thm.subgrad}
\frac{ \sum_{i=0}^{k} t_i (\varphi(x_i)+\psi(x_{i+1})) -  \frac{1}{2}\sum_{i=0}^{k} t_i^2 \|g^\varphi_i\|^2}{\sum_{i=0}^{k} t_i}
&\le   -f^*(z_k)+\ip{z_k}{x_0} - \frac{\sum_{i=0}^{k} t_i}{2}\|z_k\|^2  \\
&\le \min_{u\in\R^n}\left\{f(u) + \frac{1}{2\sum_{i=0}^{k} t_i}\|u-x_0\|^2  \right\}.\notag
\end{align}
In particular,  
\[
\frac{ \sum_{i=0}^{k} t_i (\varphi(x_i)+\psi(x_{i+1})) - \frac{1}{2} \sum_{i=0}^{k} t_i^2 \|g^\varphi_i\|^2}{\sum_{i=0}^{k} t_i }
\le   f(x) +\frac{\|x_0- x\|^2}{2\sum_{i=0}^{k} t_i}
\]
for all $ x \in \R^n.$ 
\end{proposition}
\begin{proof}
Let $\LHS_k$ and $\RHS_k$ denote respectively the left-hand and right-hand sides in~\eqref{eq.thm.subgrad}.  We proceed by induction.  
For $k=0$ we have
\begin{align*}
\LHS_0 &= \varphi(x_0)+\psi(x_1) - \frac{t_0\|g^\varphi_0\|^2}{2} 
\\&= -\varphi^*(g^\varphi_0) + \ip{g^\varphi_0}{x_0} - \psi^*(g^\psi_0) + \ip{g^\psi_0}{x_1} - \frac{t_0\|g^\varphi_0\|^2}{2} 
\\&\le -f^*(g_0) + \ip{g_0}{x_0} - \frac{t_0\|g_0\|^2}{2}
\\&= \RHS_0.
\end{align*}
The second step follows from~\eqref{eq.property.0} and $g^\varphi_0 \in \partial \varphi(x_0),\; g^\psi_0 \in \partial \psi(x_1)$.  The third step follows from~\eqref{eq.property.1} and $g_0 = g^\varphi_0+ g^\psi_0, \; x_1 = x_0 - t_0 \cdot g_0$.

Next we show the main inductive step $k$ to $k+1$.  
Observe that  $z_{k+1} = (1-\gamma_k) z_{k} + \gamma_k g_{k+1}$ for $k=0,1,\dots$
where $\gamma_k = \frac{t_{k+1}}{\sum_{i=0}^{k+1} t_i} \in (0,1)$.  
Proceeding exactly as in the proof of Theorem~\ref{main.thm} we get
\begin{align*}
\RHS_{k+1} - (1-\gamma_k)\RHS_k &\ge \gamma_k\left(\varphi(x_{k+1}) + \psi(x_{k+2}) + \ip{g^\psi_{k+1}}{x_{k+1}-x_{k+2}} - \frac{t_{k+1}\|g_{k+1}\|^2}{2} \right) \\
&= \gamma_k\left(\varphi(x_{k+1}) + \psi(x_{k+2}) + \frac{t_{k+1}\|g^\psi_{k+1}\|^2}{2} - \frac{t_{k+1}\|g^\varphi_{k+1}\|^2}{2} \right).
\end{align*}
The second step follows because $g_{k+1} = g^\varphi_{k+1}+ g^\psi_{k+1}$ and $x_{k+2} = x_{k+1} - t_{k+1} \cdot g_{k+1}.$
The proof is thus completed by observing that
\begin{align*}
\LHS_{k+1} - (1-\gamma_k)\LHS_k &= \gamma_k\left(\varphi(x_{k+1})+\psi(x_{k+2}) - \frac{t_{k+1}\|g^\varphi_{k+1}\|^2}{2}\right) 
\\&\le 
\gamma_k\left(\varphi(x_{k+1})+\psi(x_{k+2}) + \frac{t_{k+1}\|g^\psi_{k+1}\|^2}{2} - \frac{t_{k+1}\|g^\varphi_{k+1}\|^2}{2}\right).
\end{align*}

\end{proof}

Let $C \subseteq \R^n$ be a nonempty closed convex set and $\psi = I_C$.  As noted above, in this case Algorithm~\ref{algo.subgrad} becomes the projected subgradient algorithm for problem \eqref{eq.problem.const}.   
We next show that in this case Proposition~\ref{prop.subgrad} yields the classical convergence rates~\eqref{eq.classic.1} and~\eqref{eq.classic.2}, as well and the modern and more general one~\eqref{eq.modern}  recently established  by Grimmer~\cite[Theorem 5]{Grim17}. 

Suppose $\bar \varphi = \dmin_{x\in C} \varphi(x)$ is finite and $\bar X:= \{x\in C: \varphi(x) = \bar \varphi\}$ is nonempty.  From Proposition~\ref{prop.subgrad} it follows that
\begin{equation}\label{eq.subgrad.rate}
\sum_{i=0}^k t_i(\varphi(x_i) - \bar \varphi) \le \frac{\sum_{i=0}^k t_i^2 \|g^\varphi_i\|^2 + \dist(x_0,\bar X)^2}{2}. 
\end{equation}
In particular, if $\|g\| \le L$ for all $x\in C$ and $g\in \partial \varphi(x)$ then~\eqref{eq.subgrad.rate} implies
\begin{equation}\label{eq.classic.1}
\min_{i=0,\dots,k}(\varphi(x_i) - \bar \varphi) \le \frac{\sum_{i=0}^k t_i^2 L^2 + \dist(x_0,\bar X)^2}{2\sum_{i=0}^k t_i}. 
\end{equation}
Let $\alpha_i := t_i \|g^\varphi_i\|, \; i=0,1,\dots.$  Then Step 4 in Algorithm~\ref{algo.subgrad} can be rewritten as $x_{k+1} = \Pi_C\left(x_k - \alpha_k \cdot \frac{g^\varphi_k}{\|g^\varphi_k\|}\right)$ provided $\|g^\varphi_k\| > 0,$ which occurs as long as $x_k$ is not an optimal solution to \eqref{eq.problem.const}.   If $\|g^\varphi_i\| > 0$ for $i=0,1,\dots,k$ then~\eqref{eq.subgrad.rate} implies
\begin{equation}\label{eq.classic.2}
\min_{i=0,\dots,k}(\varphi(x_i) - \bar \varphi) \le L \cdot \frac{\sum_{i=0}^k \alpha_i^2 + \dist(x_0,\bar X)^2}{2\sum_{i=0}^k \alpha_i}. 
\end{equation}

Let $\mathcal L: \R_+ \rightarrow \R_+$. Following Grimmer~\cite{Grim17}, the subgradient oracle for $\varphi$ is $\mathcal L$-steep on $C$ if for all $x\in C$ and $g\in \partial \varphi(x)$
\[
\|g\| \le \mathcal L(\varphi(x) - \bar \varphi).
\]
As discussed by Grimmer~\cite{Grim17}, $\mathcal L$-steepness is a more general and weaker condition than the traditional bound $\|g\| \le L$ for all $x\in C$ and $g\in \partial \varphi(x)$.  Indeed, the latter bound is precisely $\mathcal L$-steepness for the constant function $\mathcal L(t) = L$ and holds when $\varphi$ is $L$-Lipschitz on $C$.

Suppose the subgradient oracle for $\varphi$ is $\mathcal L$-steep for some $\mathcal L: \R_+ \rightarrow \R_+$.  If $\alpha_i:= t_i \|g^\varphi_i\| > 0$ for $i=0,1,\dots,k$ then~\eqref{eq.subgrad.rate} implies
\[
\sum_{i=0}^k\alpha_i \cdot \frac{\varphi(x_i) - \bar \varphi}{\mathcal L(\varphi(x_i) - \bar \varphi)} \le \frac{\sum_{i=0}^k \alpha_i^2 + \dist(x_0,\bar X)^2}{2},
\]
and thus
\begin{equation}\label{eq.modern}
\min_{i=0,\dots,k}(\varphi(x_i) - \bar \varphi) \le \sup\left\{t: 
\frac{t}{\mathcal L(t)} \le \frac{\sum_{i=0}^k \alpha_i^2 + \dist(x_0,\bar X)^2}{2\sum_{i=0}^k \alpha_i}  \right\}.
\end{equation}

\section*{Acknowledgements}

This research has been  funded by NSF grant CMMI-1534850.

\bibliographystyle{plain}

\end{document}